\documentclass[12pt]{amsart}
\usepackage[latin1]{inputenc}
\usepackage{amsxtra,amssymb,amsthm,amsmath,amscd,mathrsfs}
\usepackage[all]{xy}
\def \C {{\mathbb C}}

\def \F {{\mathbb F}}

\def \N {{\mathbb N}}

\def \Q {{\mathbb Q}}

\def \Z {{\mathbb Z}}

\def\tr{\mathop{\rm tr}\nolimits}
\def\frakp{{\mathfrak{p}}}

\def \d {\,{\rm d}}
\def\re{{\Re e\,}}

\def\Li{\hbox{{\rm Li}}}

\def\leq{\leqslant}
\def\geq{\geqslant}
\def\le{\leqslant}
\def\ge{\geqslant}

\theoremstyle{plain}
\newtheorem{theorem}{Theorem}[section]
\newtheorem{lemma}[theorem]{Lemma}
\newtheorem{corollary}[theorem]{Corollary}
\newtheorem{conjecture}[theorem]{Conjecture}
\newtheorem{hypothesis}[theorem]{Hypothesis}

\numberwithin{equation}{section}

\begin{document}

\title[Almost prime values of the order of elliptic curves]
{Almost prime values of the order of elliptic curves over finite fields}
\author{C. David \& J. Wu}
\address{
Department of Mathematics and Statistics,
Concordia University,
1455 de Mai\-sonneuve West,
Montr\'eal, QC, H3G 1M8, Canada}
\email{cdavid@mathstat.concordia.ca}
\address{Institut Elie Cartan Nancy (IECN)
\\
Nancy-Universit\'e, CNRS, INRIA
\\
Boulevard des Aiguillettes, B.P. 239
\\
54506 Van\-d\oe uvre-l\`es-Nancy
\\
France}
\email{wujie@iecn.u-nancy.fr}
\address{School of Mathematical Sciences
\\
Shandong Normal University
\\
Jinan, Shandong 250100
\\
China}

\date{\today}

\begin{abstract}
Let $E$ be an elliptic curve over $\Q$ without complex
multiplication, and which is not isogenous to a curve with
non-trivial rational torsion. For each prime $p$ of good reduction,
let $|E(\F_p)|$ be the order of the group of points of the reduced
curve over $\F_p$. We prove in this paper that, under the GRH, there
are at least $2.778 C_E^{\rm twin} x / (\log{x})^2$ primes $p$ such
that $|E(\F_p)|$ has at most 8 prime factors, counted with
multiplicity. This improves previous results of Steuding \& Weng
\cite{SW05} and Murty \& Miri \cite{MM01}. This is also the first
result where the dependence on the conjectural constant $C_E^{\rm
twin}$ appearing in the twin prime conjecture for elliptic curves
(also known as Koblitz's conjecture) is made explicit. This is
achieved by sieving a slightly different sequence than the one of
\cite{SW05} and \cite{MM01}. By sieving the same sequence and using
Selberg's linear sieve, we can also improve the constant of Zywina \cite{zywina-largesieve}
appearing in the upper bound for the number of primes $p$ such that $|E(\F_p)|$ is prime.
Finally, we remark that our results still hold
under an hypothesis weaker than the GRH.
\end{abstract}
\maketitle

\addtocounter{footnote}{1}

\section{Introduction}

\smallskip

The twin prime conjecture is one of the oldest questions in number
theory, and can be stated as: there is an infinity of prime numbers
$p$ such that $p+2$ is also prime. The best known result is due to
Chen \cite{C73}, who proved that
\begin{equation}\label{Chen}
\big|\big\{p\le x : p+2=P_2\big\}\big|\ge 0.335 \frac{C_{\rm twin}x}{(\log x)^2}
\end{equation}
for $x\ge x_0$,
where $P_r$ denotes an integer having at most $r$ prime factors counted with multiplicity and
$$
C_{\rm twin}:=2\prod_{\ell>2}\bigg(1-\frac{1}{(\ell-1)^2}\bigg).
$$
Here and in the sequel, the letters $p$ and $\ell$ denote prime
numbers. There are many generalisations of the twin prime
conjecture, and in particular, an  analogous conjecture for elliptic
curves was formulated by Koblitz \cite{Ko88}. Let $E$ be an elliptic
curve defined over $\Q$ and denote by $E(\F_p)$ the reduction of $E$
modulo $p$. Koblitz \cite{Ko88} conjectured that if $E$ is not
$\Q$-isogenous to an elliptic curve with non-trivial $\Q$-torsion
subgroup then
\begin{equation}\label{KoblitzC}
\pi_E^{\rm twin}(x)
:= \big|\big\{p\le x : |E(\F_p)|\;\hbox{is prime}\big\}\big|
\sim \frac{C_E^{\rm twin}x}{(\log x)^2}
\end{equation}
with a non-zero constant  $C_E^{\rm twin}$ as $x\to\infty$,
where the constant $C_E^{\rm twin}$ can be defined as an Euler product as the twin prime constant
$C_{\rm twin}$.
It will be described explicitly in
Section \ref{koblitzconstant}. This conjecture has theoretical
relevance to elliptic curve cryptosystems based the discrete
logarithm problem in the group $E(\F_p)$.

As the twin prime conjecture, Koblitz's conjecture is still open,
but was shown to be true on average over all elliptic curves
\cite{BCD08}. One can also apply sieve methods to get lower bounds
for the number of primes $p$ such that $|E(\F_p)|$ is almost-prime.
It is necessary to distinguish two cases: when $E$ has complex
multiplication (CM) or not.
In the first case, Iwaniec \& Jim\'enez Urroz \cite{IU07, IU08}
have obtained an analogue of Chen's theorem
(\ref{Chen}). In the non-CM case, all results assume the generalized
Riemann hypothesis (GRH) for Dedekind zeta-functions  of some number fields.
The first result of this type is due to Miri \& Murty
\cite{MM01}, who proved by using Selberg's sieve \cite{Bom87} that
\begin{equation}\label{MM}
\big|\big\{p\le x : |E(\F_p)|=P_{16}\big\}\big|\gg \frac{x}{(\log
x)^2}
\end{equation}
for $x\ge x_0(E)$, where the implicit constant depends on the
elliptic curve $E$. Recently Steuding and Weng \cite{SW05} have
improved 16 to 9, by using Richert's logarithmic weighted sieve
\cite{HR74} and some improvements to the error term of the explicit
Chebotarev  Density Theorem due to Serre \cite{Serre81} and M.-R. Murty,
V.-K. Murty \& Saradha \cite{MMS88}.

We prove in this paper a better result under a weaker hypothesis,
namely we replace the GRH by the $\theta$-hypothesis which states
that there are no zeroes with $\re s> \theta$ for Dedekind
zeta-functions and Artin $L$-functions. This is  stated in Section
\ref{CDT} as Hypothesis \ref{theta}. We also write explicitly the
constant in terms of the twin prime constant $C_E^{\rm twin}$ by
modifying slightly the set to sieve.

\begin{theorem}\label{CTEllipticCurve1}
Let $E$ be an elliptic curve over $\Q$ without complex
multiplication which is not $\Q$-isogenous to an elliptic curve with
non-trivial $\Q$-torsion. Assuming Hypothesis \ref{theta} for any
$1/2\le \theta<1$, we have
\begin{equation}\label{1}
\big|\big\{p\le x : (|E(\F_p)|, M_E)=1, \; |E(\F_p)|=P_r\big\}\big|
\ge \frac{1.323}{1-\theta}\frac{C_E^{\rm twin}x}{(\log x)^2}
\end{equation}
for $x\ge x_0(E)$, where
$M_E$ is an integer depending on $E$
(which will be described explicitly in Section \ref{koblitzconstant}) and
\begin{equation}\label{defr}
r = r(\theta) := \bigg[\frac{18+2\theta}{5(1-\theta)}\bigg] + 1.
\end{equation}
Here $[t]$ denotes the integral part of $t$.
\end{theorem}

Since $(18+2\theta)/(5-5\theta)<8$ if and only if $\theta<11/21$, we
immediately obtain the following result, which improves, under a
weaker hypothesis, the result of Steuding \& Weng mentioned above.

\begin{corollary}\label{CTEllipticCurve1b}
Let $E$ be an elliptic curve over $\Q$ without complex
multiplication which is not $\Q$-isogenous to an elliptic curve with
non-trivial $\Q$-torsion. Assuming Hypothesis \ref{theta} for any
$1/2\le \theta<11/21$, we have
\begin{equation}\label{1b}
\big|\big\{p\le x : (|E(\F_p)|, M_E)=1, \; |E(\F_p)|=P_8\big\}\big|
\ge 2.778\frac{C_E^{\rm twin}x}{(\log x)^2}
\end{equation}
for $x\ge x_0(E)$.
In particular (\ref{1b}) holds if we assume GRH.
\end{corollary}

Of course, Theorem \ref{CTEllipticCurve1} and Corollary
\ref{CTEllipticCurve1b} imply the same lower bound for
$$\big|\big\{p\le x \; : \; |E(\F_p)|=P_8\big\}\big|$$
since we are getting a lower bound for a smaller set.
We will see in Section \ref{koblitzconstant} that
it is natural to count primes such that $(|E(\F_p)|, M_E)=1$ when sieving to get the
right constant in Theorem \ref{CTEllipticCurve1}.

\smallskip

Upper bounds for $\pi_E^{\rm twin}(x)$ were first studied by
Cojocaru  who showed in \cite{Co05} that $\pi_E^{\rm twin}(x)\ll
x/(\log x)^2$ by using Selberg's linear sieve under the GRH. The
implicit constant depends on the conductor of $E$, but the exact
dependency was not worked out. Very recently, Zywina
\cite{zywina-largesieve} applied an abstract form the large sieve to
obtain that
\begin{equation}\label{ZywinaUB}
\pi_E^{\rm twin}(x)
\le \{22+o(1)\} \frac{C_E^{\rm twin}x}{(\log x)^2}
\end{equation}
as $x\to\infty$. His result applies to a more general form of
Koblitz's conjecture, where the elliptic curve $E$ can be defined
over any number field, and can have rational torsion on an isogenous
curve.

The second aim of this paper is to show that Selberg's linear sieve
allows us to obtain the correct twin prime constant $C_E^{\rm twin}$
with a better constant factor than (\ref{ZywinaUB}) in the case of
elliptic curves over $\Q$ with no rational torsion.

\begin{theorem}\label{Upperbound}
Under the condition of Theorem \ref{CTEllipticCurve1},
for any $\varepsilon>0$ we have
$$
\pi_E^{\rm twin}(x)
\le \bigg(\frac{5}{1-\theta}+\varepsilon\bigg)\frac{C_E^{\rm twin}x}{(\log x)^2}
$$
for $x\ge x_0(E, \varepsilon)$.
\end{theorem}

Then, assuming GRH, Theorem \ref{Upperbound} allows us to improve
the constant in (\ref{ZywinaUB}) from 22 to 10.

\medskip

{\bf Acknowledgments}.
The second author wishes to thank l'Institut \'Elie Cartan de Nancy-Universit\'{e}
for hospitality and support
during the preparation of this article.

\vskip 10mm

\section{Koblitz's Conjecture}
\label{koblitzconstant}

Let $E$ be an elliptic curve over $\Q$ without complex
multiplication with conductor $N_E$, and let $L_n := \Q(E[n])$ be
the field extension obtained from $\Q$ by adding the coordinates
of the points of $n$-torsion to $\Q$. This is a Galois extension of
$\Q$, and in all this paper, we denote $$G(n) =
\mbox{Gal}(L_n/\Q).$$ Since $E[n](\bar{\Q}) \simeq \Z/n \Z \times \Z/n \Z$,
choosing a basis for the $n$-torsion and looking at the
action of the Galois automorphisms on the $n$-torsion, we get an
injective homomorphism
$$
\rho_n : G(n) \hookrightarrow \mbox{GL}_2(\Z/n \Z).
$$
If $p \nmid nN_E$, then $p$ is unramified in $L_n/\Q.$ Let $p$ be an
unramified prime, and let $\sigma_p$ be the Artin symbol of $L_n/\Q$
at the prime $p$.  For such a prime $p$, $\rho_{n}(\sigma_p)$ is a
conjugacy class of matrices of $\mbox{GL}_2(\Z/n \Z)$. Since the
Frobenius endomorphism $(x,y) \mapsto (x^p, y^p)$ of $E$ over $\F_p$
satisfies the polynomial $x^2 - a_p x + p$ where $a_p$ is defined by
the relation
\begin{equation}\label{defap}
|E(\F_p)|=p+1-a_p,
\end{equation}
it is not difficult to see that
$$
\tr(\rho_{n}(\sigma_p)) \equiv a_p\,({\rm mod}\,n) \qquad{\rm
and}\qquad \det(\rho_{n}(\sigma_p)) \equiv p\,({\rm mod}\,n).
$$

It was shown by Serre \cite{Serre72} that the image of the Galois
representations $\rho_n$ are as big as possible, and there exists a
positive integer $M_E$ such that
\begin{eqnarray}
\label{serre1} &&\mbox{If $(n, M_E)=1$, then $G(n) =
\mbox{GL}_2(\Z/n \Z)$;}
\\
\label{serre2} &&\mbox{If $(n, M_E)=(n,m)=1$, then $G(mn) \simeq
G(m) \times G(n)$.}
\end{eqnarray}

For any integer $n$, let $$C(n) = \left\{ g \in G(n) \;:\;
\det(g)+1-\tr(g)\equiv 0\,({\rm mod}\,n) \right\}.$$ The original
Koblitz constant was defined in terms of the local probabilities for
the event $\ell \nmid p+1-a_p(E)$, which can  be evaluated by
counting matrices $g$ in $\mbox{GL}_2(\Z / \ell \Z)$. More
precisely, for each prime $\ell$, the correcting probability factor
is the quotient
\begin{eqnarray}
\label{eulerfactors} E(\ell) = \frac{\displaystyle{1
-|\C(\ell)|/|G(\ell)|}}{\displaystyle {1 - 1/\ell}}
\end{eqnarray}
where the numerator is the probability that $p+1-a_p(E)$ is not
divisible by $\ell$ and the denominator is the probability that a
random integer is not divisible by $\ell$. If
$G(\ell)=\mbox{GL}_2(\Z/\ell\Z)$, which happens for all but finitely
many primes by (\ref{serre1}), then
\begin{eqnarray*}
E(\ell) = \frac{\displaystyle{1
-|\C(\ell)|/|G(\ell)|}}{\displaystyle {1 - 1/\ell}} &=& 1 -
\frac{\ell^2-\ell-1}{(\ell-1)^3 (\ell+1)}.
\end{eqnarray*}
The constant $C_E^{\rm twin}$ of \cite{Ko88} was defined as the
product over all primes $\ell$ of the Euler factors $E(\ell)$. In
\cite{zywina}, Zywina made the observation that the probabilities
are not multiplicative, as the events are not independent: the
fields $\Q(E[\ell])$ are never disjoint for all primes $\ell$, as
observed by Serre in \cite{Serre72}. For any integer $m$, let
$$\Omega(m) = \left\{ g \in G(m):  (\det(g)+1-\tr(g),m) \neq 1  \right\}.$$
According to the refinement of \cite{zywina}, the probability factor
at $M_E$ is defined as $$\frac{1 -
|\Omega(M_E)|/|G(M_E)|}{\prod_{\ell\mid M_E} (1-1/\ell)}.$$
The twin prime constant $C_E^{\rm twin}$ is then defined as
\begin{equation}\label{defCEtrwin}
C_E^{\rm twin}
:=  \frac{1 - |\Omega(M_E)|/|G(M_E)|}
{\prod_{\ell\mid M_E} (1-1/\ell)} \prod_{\ell \nmid M_E}
\bigg(1-\frac{\ell^2-\ell-1}{(\ell-1)^3(\ell+1)}\bigg).
\end{equation}
We also remark that if we assume Koblitz's conjecture, it follows
immediately that
\begin{equation}\label{KoblitzCModified}
\big|\big\{p\le x : (|E(\F_p)|, M_E)=1, \; |E(\F_p)|\;\hbox{is
prime}\big\}\big| \sim \frac{C_E^{\rm twin}x}{(\log x)^2}
\end{equation}
as $x\to\infty$. It is easy to see that (\ref{KoblitzC}) and
(\ref{KoblitzCModified}) are equivalent, as we can write
\begin{equation}\label{Compare}
\sum_{\substack{p\le x\\ |E(\F_p)|\,{\rm is\;prime}}} 1
= \sum_{\substack{p\le x\\ |E(\F_p)|\,{\rm is\;prime}\\ (|E(\F_p)|, M_E)=1}} 1
+ \sum_{\substack{p\le x\\ |E(\F_p)|\,{\rm is\;prime}\\ (|E(\F_p)|, M_E)>1}} 1.
\end{equation}
The condition that $|E(\F_p)|\;{\rm is\;prime}$ and $(|E(\F_p)|,
M_E)>1$ implies that $|E(\F_p)|=p'$ for some prime $p'$ which
divides $M_E$. On the other hand, the relation (\ref{defap}) and
Hasse's bound $|a_p|\leq 2\sqrt{p}$ allow us to deduce that
$|E(\F_p)|\ge p+1-2\sqrt{p}\ge p/16$. Thus $p\le 16p'\le 16M_E$.
This shows that the second sum on the right-hand side of
(\ref{Compare}) is bounded by $16M_E$.

In this paper, we sieve the sequence
\begin{equation}\label{Amodified}
\{|E(\F_p)| : p\le x, \; (|E(\F_p)|, M_E)=1\}
\end{equation}
instead of the sequence
\begin{equation}\label{A}
\{|E(\F_p)| : p\le x\}
\end{equation}
suggested by (\ref{KoblitzC}) as in \cite{MMS88, SW05, SW05bis}.
This will allow us to  obtain the correct twin prime constant
$C_E^{\rm twin}$ in our theorem.

We will need later the fact that
\begin{eqnarray}
\label{ProbME} {1 - {\displaystyle\frac{|\Omega(M_E)|}{|G(M_E)|}}} =
\sum_{d \mid M_E} \mu(d) \frac{|C(d)|}{|G(d)|}. \end{eqnarray} This
can be proven by using the Chebotarev Density Theorem in the
extension $L_{M_E}/\Q$. We are using here only the density result of
the Chebotarev Density Theorem, and we refer the reader to Section
\ref{CDT} for versions of the Cheboratev Density Theorem with an
explicit error term that will be needed to perform the sieve. Since
$1 -  |\Omega(M_E)|/|G(M_E)|$ is the proportion of matrices in $g
\in G(M_E)$ with $(\det{g}+1-\tr{g}, M_E)=1$, we have that
\begin{align*}
\left( 1 -  \frac{|\Omega(M_E)|}{|G(M_E)|} \right)
\pi(x)
&\sim \sum_{\substack{p \leq x\\ (|E(\F_p)|, M_E)=1}} 1 \\ &=
\sum_{d \mid M_E} \mu(d) \sum_{\substack{p \leq x\\ d \mid |E(\F_p)|}} 1
\\
&\sim \pi(x) \sum_{d \mid M_E} \mu(d) \frac{|C(d)|}{|G(d)|},
\end{align*}
as $x\to\infty$.

\section{Chebotarev Density Theorem}
\label{CDT}

We write in this section an explicit Chebotarev Density Theorem
associated with the Galois extensions of $\Q$ obtained by adding the
coordinates of the points of $n$-torsion to $\Q$.

We first need some notation and definitions. In all this section,
let $L/K$ be a finite Galois extension of number fields with Galois
group $G$, and let $C$ be a union of conjucacy classes in $G$. Let
$n_K$ and $n_L$ be the degrees of $K$ and $L$ over $\Q$, and $d_K$
and $d_L$ their absolute discriminant. Let
$$
M(L/K) := |G| d_K^{1/n_K} \prod_{p} p,
$$
where the product is over the rational primes $p$ which lie below
the ramified primes of $L/K$. Let $\pi_C(x, L/K)$ be the number of
prime ideals $\frakp \in K$ such that $N\frakp \leq x$ which are
unramified in $L/K$ and with $\sigma_\frakp \in C$, where
$\sigma_\frakp$ is the Artin symbol at the prime ideal $\frakp$.

The following Theorem is an effective version of the Chebotarev
Density Theorem due to Lagarias and Odlyzko \cite{LaOd}, with a
slight refinement due to Serre \cite{Serre81}.

\begin{theorem}[Effective Chebotarev Density Theorem]\label{Chebotarev}
{\rm (i)}
Let $\beta$ be the exceptional zero of the Dedekind
zeta function of $L$ (if any). Then, for all $x$ such that $\log{x} \gg n_L (\log{d_L})^2$,
$$\pi_C(x, L/K) = \frac{|C|}{|G|} \Li(x) + O \left( \frac{|C|}{|G|}
\Li(x^\beta) + |\tilde{C}| x \exp\left\{-c n_L^{-1/2} \log^{1/2}{x}\right\}
\right),$$ where $c$ is a positive absolute constant, and $|\tilde{C}|$ is
the number of conjugacy classes in $C$.
\par
{\rm (ii)}
Assuming the GRH for the Dedekind zeta function of
the number field $L$, we have that
$$
\pi_C(x, L/K) = \frac{|C|}{|G|} \Li(x) + O \left( x^{1/2} |C| n_K
\log{\big(M(L/K)x\big)} \right).
$$
\end{theorem}

\goodbreak

We now apply Theorem \ref{Chebotarev} to the torsion fields of
elliptic curves.

\begin{theorem}\label{CDT-notimproved}
Let $E$ be an elliptic curve over $\Q$ without complex
multiplication, and let $L_n, C(n)$ and $G(n)$ be as defined above.
Assuming the GRH for the Dedekind zeta function of the number fields
$L_n$, we have that
\par
{\rm (i)}
Let $d$ be a square-free integer such that $(d, M_E)=1$. Then,
$$
\pi_{C(d)}(x, L_d/\Q) = \left( \prod_{\ell \mid d}
\frac{\ell^2-2}{(\ell-1)(\ell^2-1)} \right) \Li(x) + O\left(d^3
x^{1/2} \log{(d N_E x)}\right).
$$
\par
{\rm (ii)}
Let $\ell$ be a prime such that $\ell \nmid M_E$. Then,
$$
\pi_{C(\ell^2)}(x, L_{\ell^2}/\Q) =
\frac{\ell^3-\ell-1}{\ell^2(\ell^2-1)(\ell-1)} \Li(x) + O\left(\ell^6 x^{1/2}\log{(\ell N_E x)}\right).
$$
\end{theorem}

It was noticed by Serre \cite{Serre81} that one can improve
significantly the error term of Theorem \ref{Chebotarev}.(ii)
(basically replacing $|C|$ with $|C|^{1/2}$) by writing the
characteristic functions of the conjucacy classes $C$ in terms of
the irreducible characters of $G$, and then working with the Artin's
$L$-functions associated with those characters.
 Further applications
can be found in \cite{MMS88}, \cite{CoDa} and in \cite{SW05} for the
present application. For the convenience of the reader, we summarize
in the next two theorems the main features of the approach, and how
it can be applied in our case.

We first define some notation. For each character $\chi$ of $G$, let
$L(s, \chi)$ be the Artin $L$-function associated to $\chi$. If $G$ is
an abelian group, the Artin $L$-functions of $L/K$ corresponds to
Hecke $L$-functions, and are then analytic functions of the complex
plane. In general, we have

\begin{conjecture}[Artin's conjecture] Let $\chi$ be an irreducible
non-trivial character of $G$. Then, $L(s, \chi)$ is analytic in the
whole complex plane.
\end{conjecture}

We will write the improvement of  Theorem \ref{Chebotarev} under the
$\theta$-hypothesis for the zeros of the $L$-functions in the
critical strip, where $1/2 \leq \theta <1$, and not the full Riemann
Hypothesis, which allows us to obtain an improvement of the results
of \cite{SW05} under a weaker hypothesis.

\begin{hypothesis}[$\theta$-hypothesis] \label{theta} Let $L(s)$ be a Dedekind zeta function, or an Artin
$L$-function satisfying Artin's conjecture. Let $1/2 \leq \theta < 1$.
Then $L(s)$ is non-zero for $\re s > \theta$.
\end{hypothesis}

Let $\varphi$ be a class function on $G$, i.e. a function which is
invariant under conjugation.
Define
$$
\pi_{\varphi}(x, L/K)
:= \sum_{\substack{N \frakp \leq x\\ {\frakp \; \text{unramified}}}}
\varphi(\sigma_\frakp).
$$
If $C$ is a conjugacy class (or a union of conjugacy classes) in
$G$, and $1_C$ is its characteristic function, then $\pi_{1_C}(x,
L/K) = \pi_C(x, L/K)$ as defined above.

To define $\tilde{\pi}_{\varphi}(x, L/K)$, we need to extend the
definition of the Artin symbol $\sigma_\frakp$ at the ramified
primes $\frakp$. This is done in \cite[Section 2.5]{Serre81}, and we
refer the reader to this paper. Then,
$$
\tilde{\pi}_{\varphi}(x, L/K) = \sum_{{N \frakp^m \leq x}}
\frac{1}{m} \varphi(\sigma_\frakp^m),
$$
where the sum runs over all pairs of primes $\frakp$ of $K$ and
integers $m \geq 1$ such that $N \frakp^m \leq x$. With this
definition, if $\varphi=\chi$ is a character of $G$ and $L(s, \chi)$
is the Artin $L$-function of $\chi$ with
$$
\log{L(s, \chi)} = \sum_{n=1}^\infty a_n(\chi) n^{-s},
$$
then
$$
\tilde{\pi}_{\chi}(x, L/K) = \sum_{n \leq x} a_n(\chi).
$$
Then, $\tilde{\pi}_{\varphi}(x, L/K)$ has the following two
important properties.

\begin{lemma}\cite[Propostion 7]{Serre81}
Under the previous notation, we have
\label{propone}
$$
{\pi}_{\varphi}(x, L/K) = \tilde{\pi}_{\varphi}(x, L/K) + O \left(
\sup_{g \in G} | \varphi(g) | \left( \frac{1}{|G|} \log{d_L} + n_K
x^{1/2} \right) \right) .
$$
\end{lemma}

\begin{lemma}\cite[Propostion 8]{Serre81}
\label{proptwo} Let $H$ be a subgroup of $G$, $\varphi_H$ a class
function on $H$ and $\varphi_G = {\rm Ind}_H^G(\varphi_H)$. Then
$$
{\pi}_{\varphi_G}(x, L/K) = \tilde{\pi}_{\varphi_H}(x, L/L^H) .
$$
\end{lemma}
\begin{proof} This follows from the invariance of the Artin
$L$-functions of induced characters. \end{proof}

Using lemmas \ref{propone} and \ref{proptwo}, we deduce that

\begin{theorem} \label{theoreminduced}
Let $L/K$ be a Galois extension of number fields
with Galois group $G$. Let $H$ be a subgroup of $G$ and $C$ a
conjugacy class in $G$ such that $C \cap H \neq \emptyset.$ Let
$C_H$ be the conjucacy class in $H$ generated by $C \cap H$. Then
\begin{align*}
\pi_C(x, L/K)
& = \frac{|H|}{|G|}  \frac{|C|}{|C_H|} \pi_{C_H}(x, L/L^H)
\\
& \quad
+ O \left( \frac{|C|}{|C_H| |G|} \log{d_L} + \frac{|H|}{|G|}
\frac{|C|}{|C_H|} [L^H: \Q] x^{1/2} + [K:\Q] x^{1/2} \right).
\end{align*}
\end{theorem}
\begin{proof} Let $\varphi$ be the class function on $G$ induced
from $C_H$. It is easy to see that
$$
\varphi = \mbox{Ind}_H^G(1_{C_H}) = \frac{|G|}{|H|}
\frac{|C_H|}{|C|} 1_C,
$$
and by Lemma \ref{proptwo}
$$
 \tilde{\pi}_{C_H}(x, L/L^H) =  \frac{|G|}{|H|}
\frac{|C_H|}{|C|} \pi_C(x, L/K).
$$
Using Lemma \ref{propone} to bound the difference between $\pi$ and
$\tilde{\pi}$, we get the result.
\end{proof}

The second piece needed for the improved Chebotarev Density Theorem
is an estimate for $\pi_C(x, L/K)$ in the case that $G$ has a normal
subgroup $H$ with the property that the Artin $L$-functions of $G/H$
satisfy Artin's conjecture and the $\theta$-hypothesis. For
$\theta=1/2$, the following theorem is Proposition 3.12 from
\cite{MMS88}.

\begin{theorem} \label{theoremartin}
Let $D$ be a non-empty union of conjugacy classes in
$G$ and let $H$ be a normal subgroup of $G$ such that for all Artin
$L$-functions attached to characters of $G/H \simeq {\rm Gal}(L^H/K)$,
the Artin conjecture and the $\theta$-hypothesis
hold. Suppose also that $HD \subseteq D$. Then,
$$
\pi_D(x, L/K) = \frac{|D|}{|G|} \Li(x)
+ O \bigg( \left(\frac{|D|}{|H|} \right)^{1/2} x^\theta n_K \log{\big(M(L/K)x\big)} \bigg).
$$
\end{theorem}

We now apply the last two theorems to get an improvement to Theorem
\ref{CDT-notimproved}. Let $E$ be an elliptic curve without complex
multiplication, let $\ell \nmid M_E$, and let $L_\ell/\Q$, $G(\ell)$
and $C(\ell)$ be as defined above. Let $B(\ell) \subset G(\ell)$ be
the subgroup of Borel matrices. Let $C_B(\ell)$ be the union of
conjugacy classes generated by $B(\ell) \cap C(\ell)$.
Applying Theorem \ref{theoreminduced}, we get
\begin{equation}\label{firstEQ}
\pi_{C(\ell)}(x, L_\ell/\Q) = \frac{|B(\ell)|}{|G(\ell)|}
\frac{|C(\ell)|}{|C_B(\ell)|} \pi_{C_{B(\ell)}}\big(x,
L_\ell/L_\ell^{B(\ell)}\big) + O \left( \ell \log{(\ell N_E)} + \ell
x^{1/2} \right)
\end{equation}
using the bounds of \cite[Section 1.4]{Serre81} for $\log{d_L}$.

Let $U(\ell) \subset B(\ell)$ be the subgroup of unipotent matrices.
It is easy to see that $U(\ell)$ is a normal subgroup of $B(\ell)$,
and that $B(\ell)/U(\ell)$ is the abelian group of diagonal matrices
over $\F_\ell$. Artin's conjecture then holds for all $L$-functions
of $L_\ell^{U(\ell)}/L_\ell^{B(\ell)}$, and  we apply Theorem
\ref{theoremartin} with $G=B(\ell)$, $H=U(\ell)$ and $D=C_B(\ell)$
under the $\theta$-hypothesis for the appropriate $L$-functions.
This gives
\begin{equation}\label{secondEQ}
\pi_{C_B(\ell)}\big(x, L_\ell/L_\ell^{B(\ell)}\big) =
\frac{|C_B(\ell)|}{|B(\ell)|} \mbox{Li}(x) + O \left(  \ell^{3/2}
x^\theta \log{(\ell N_E x)} \right).
\end{equation}

We are now ready to state the improvement to Theorem
\ref{CDT-notimproved}. In the next theorem, all error terms depend
on the elliptic curve $E$. We remark that we need a version of the
Cheboratev Density Theorem in the extension $L_{n}$ where $n$ is not necessarily co-prime
to $M_E$
in order to sieve the sequence of (\ref{Amodified}).

\begin{theorem}\label{CDT-improved}
Let $E$ be an elliptic curve over $\Q$ without complex
multiplication. Assuming the $\theta$-hypothesis for the
Dedekind zeta functions of the number fields $L_n/\Q$, we have

{\rm (i)}
Let $d, m$ be a square-free integers such that
$(d, M_E)=1$ and $m \mid M_E$. Then,
$$
\pi_{C(dm)}(x, L_{dm}/\Q) =  \frac{|C(m)|}{|G(m)|} \left( \prod_{\ell
\mid d} \frac{\ell^2-2}{(\ell-1)(\ell^2-1)} \right) \Li(x) + O_E
\left(d^{3/2}x^{\theta} \log{(d  x)}\right).
$$

{\rm (ii)}
Let $\ell$ be a prime such that $\ell \nmid M_E$. Then,
$$
\pi_{C(\ell^2)}(x, L_{\ell^2}/\Q) =
\frac{\ell^3-\ell-1}{\ell^2(\ell^2-1)(\ell-1)} \Li(x) +
O_E\left(\ell^3x^{\theta}\log{(\ell x)}\right).
$$
\end{theorem}

\begin{proof}
The proof of (i) with $d = \ell$ and $m=1$ follows directly by replacing (\ref{secondEQ}) in
(\ref{firstEQ}), and the general case of (i) follows by applying the
same reasoning as above to the extension $L_{dm}/\Q$
with Galois group $G(dm) \simeq G(m) \times \prod_{\ell \mid d}
\mbox{GL}_2(\F_\ell)$. The proof of (ii) follows similarly using $G(\ell^2) \simeq \mbox{GL}_2(\Z / \ell^2
\Z)$.
\end{proof}

\vskip 10mm

\section{Greaves' weighted sieve and proof of Theorem \ref{CTEllipticCurve1}}\label{GreavesLB}

\smallskip

We first recall the simplified version of Greaves' weighted sieve of
dimension 1, i.e. taking $E=V$ and $T=U$ in \cite[Theorem A]{HR85}.

Let ${\mathcal A}$ be a finite sequence of integers and ${\mathcal
P}$ a set of prime numbers. Let ${\mathcal B}={\mathcal
B}({\mathcal P})$ denote the set of all positive square-free
integers supported on the primes of ${\mathcal P}$. For each $d\in
{\mathcal B}$, define
$$
{\mathcal A}_d:=\{a\in {\mathcal A} : a\equiv 0\,({\rm mod}\,d)\}.
$$
We assume that ${\mathcal A}$ is well distributed over arithmetic progressions $0\,({\rm mod}\,d)$
in the following sense:
There are a convenient approximation $X$ to $|{\mathcal A}|$
and a multiplicative function $w(d)$ on ${\mathcal B}$ verifying
$$
0\le w(p)<p
\qquad(p\in {\mathcal P})
\leqno(A_0)
$$
such that
\par
(i) the ``remainders''
\begin{equation}\label{defrAd}
r({\mathcal A}, d)
:= |{\mathcal A}_d| - \frac{w(d)}{d}X
\qquad(d\in {\mathcal B})
\end{equation}
are small on a average over the divisors $d$ of
\begin{equation}\label{defPz}
P(z)
:= \prod_{p<z, \, p\in {\mathcal P}} p;
\end{equation}

\par

(ii)
there exists a constant $A\ge 1$ such that
$$
\bigg|\sum_{\substack{z_1\le p<z_2\\ p\in {\mathcal P}}} \frac{w(d)}{d}\log p - \log\frac{z_2}{z_1}\bigg|
\le A
\qquad
(2\le z_1\le z_2).
\leqno(\Omega_1)
$$

\smallskip

Let $U$ and $V$ be two constants verifying
\begin{equation}\label{GreavesCond1}
V_0\le V\le {\textstyle \frac{1}{4}},
\qquad
{\textstyle \frac{1}{2}}\le U<1,
\qquad
U+3V\ge 1,
\end{equation}
where $V_0 := 0.074368\cdots$.
The simplified version of Greaves' weighted sieve function is given by
\begin{equation}\label{defH}
H({\mathcal A}, D^V, D^U)
:= \sum_{a\in {\mathcal A}} {\mathscr G}\big((a,P(D^U))\big),
\end{equation}
where $D\ge 2$ is the basic parameter of considered problem,
$$
{\mathscr G}(n) := \Big\{1 - \sum_{p\mid n, \, p\in {\mathcal P}} \big(1 - {\mathscr W}(p)\big)\Big\}^+,
$$
with $\big(\{x\}^+ := \max\{0, \, x\}\big)$,
and
$${\mathscr W}(p) := \begin{cases}
\displaystyle \frac{1}{U-V} \bigg(\frac{\log p}{\log D}-V\bigg)
& \text{if $D^{V}\le p<D^U$},
\\\noalign{\smallskip}
0
& \text{otherwise}.
\end{cases}$$
It is clear that
\begin{align*}
H({\mathcal A}, D^V, D^U) & = \sum_{\substack{a\in {\mathcal A}\\
(a, P(D^V))=1}} {\mathscr G}\big((a,P(D^U))\big)
\\
& \ge \sum_{\substack{a\in {\mathcal A}\\ (a, P(D^V))=1}}
\Big\{1 - \sum_{p\mid (a,P(D^U))} \big(1 - {\mathscr W}(p)\big)\Big\}
\\
& = \sum_{\substack{a\in {\mathcal A}\\ (a, P(D^V))=1}}
\bigg\{1 - \frac{U}{U-V}\sum_{\substack{D^V\le p<D^U\\ p\mid a}}
\bigg(1 - \frac{1}{U}\frac{\log p}{\log D}\bigg)\bigg\}.
\end{align*}
The last quantity is the sum of weights of Richert's logarithmic
weighted sieve \cite[Chapter 9, (1.2)]{HR74}. Therefore, Greaves'
weighted sieve is always better than Richert's sieve. It is worth
pointing out that Richert's logarithmic weighted sieve would have been
sufficient for our propose. In fact, for our choice of parameters,
these two sieves coincide in the main term (comparing \cite[Lemma
9.1]{HR74} and (\ref{GWeight}) below). The greatest advantage of
Greaves' weighted sieve is the bilinear form error term.
In many applications (for example $P_2$ in short intervals, \cite{Wu92}),
this advantage allows to take a larger level of distribution $D$
to obtain better results. The actual version of Chebotarev
Density Theorem does not allows us to profit of this advantage for
our problem.

As usual, let $\Omega(n)$ and $\omega(n)$ denote the number of prime
factors of $n$ counted with and without multiplicity, respectively.
Define
$$\omega(a,z)
:=\omega(a) +\sum_{\substack{p\ge z,\,\nu\ge 2\\ p^\nu\mid a}} 1,$$
where the sum is taken over all pairs of primes $p \geq z$ and
integers $\nu \geq 2$ such that $p^\nu$ divides $a$.

The function $H$ will be used to detect the integers in ${\mathcal
A}$ having few of prime factors in the following way.

\begin{lemma}\label{Greaves1}
Let $E$ be an elliptic curve over $\Q$ and let
\begin{align*}
{\mathscr A}
& :=\{|E(\F_p)| : p\le x, \; (|E(\F_p)|, M_E)=1\},
\\\noalign{\vskip 1mm}
{\mathscr P}
& :=\{p : p\nmid M_E\}.
\end{align*}
If there are real positive constants $U, V, \xi, x_0(E), B$ and
positive integer $r$ such that (\ref{GreavesCond1}) holds and for
$x\ge x_0(E)$,
\begin{align}
\max_{a\in {\mathscr A}} a
& \le D^{rU+V},
\label{GreavesCond2}
\\\noalign{\vskip 2,5mm}
\sum_{D^V\le p<D^U} |{\mathscr A}_{p^2}|
& \ll_E \frac{x}{(\log x)^3},
\label{GreavesCond3}
\\
H\big({\mathscr A}, D^V, D^U\big)
& \geq B \frac{C_E^{\rm twin} x}{(\log x)^2},
\label{GreavesCond4}
\end{align}
where $D:=x^{\xi}$,
then we have
\begin{equation}\label{GreavesR1}
\big|\big\{p\le  x : (|E(\F_p)|, M_E)=1, \; \Omega(|E(\F_p)|)\leq r\big\}\big| \geq \{B + o(1)\}
\frac{C_E^{\rm twin}x}{(\log x)^2}
\end{equation}
for $x\ge x_0(E)$.
\end{lemma}

\begin{proof}
Since ${\mathscr W}(p)\le 1$, we have $0\le {\mathscr G}(n)\le 1$ for all $n\in \N$.
Thus hypothesis (\ref{GreavesCond4}) allows us to write
\begin{equation}\label{GreavesCond5}
\begin{aligned}
\sum_{\substack{a\in {\mathscr A}\\ (a, P(D^V))=1, \, {\mathscr G}((a, P(D^U)))>0}} 1
& \ge \sum_{\substack{a\in {\mathscr A}\\ (a, P(D^V))=1}} {\mathscr G}\big((a, P(D^U))\big)
\\
& = H\big({\mathscr A}, D^V, D^U\big)
\\\noalign{\vskip 1,5mm}
& \ge B \frac{C_E^{\rm twin} x}{(\log x)^2}
\end{aligned}\end{equation}
for $x\ge x_0(E)$.

For each $a\in {\mathscr A}$ counted in the sum on the left-hand side of (\ref{GreavesCond5}),
we have $(a, P(D^V))=1$ and
\begin{align*}
0
& < \Big\{1 - \sum_{p\mid a, \; p<D^U} \big(1 - {\mathscr
W}(p)\big)\Big\}^+
\\
& = 1-\frac{1}{U-V}\sum_{p\mid a, \; p<D^U}\bigg(U-\frac{\log p}{\log D}\bigg).
\end{align*}
This and hypothesis (\ref{GreavesCond2}) imply
\begin{align*}
0
& < U-V-\sum_{p\mid a, \, p<D^U} \bigg(U-\frac{\log p}{\log D}\bigg)
\\
& \le U-V
-\sum_{p\mid a} \bigg(U-\frac{\log p}{\log D}\bigg)
-\sum_{p\ge D^U, \; \nu\ge 2, \; p^\nu\mid a} \bigg(U-\frac{\log p}{\log D}\bigg)
\\\noalign{\vskip -1mm}
& \le U-V-U\omega(a, D^U)+\frac{\log a}{\log D}
\\\noalign{\vskip 2mm}
& \le U-V-U\omega(a, D^U) + rU + V
\\\noalign{\vskip 3mm}
& = U\big(r+1-\omega(a, D^U)\big).
\end{align*}
Hence for such $a$ we have $(a, P(D^V))=1$ and $\omega(a, D^U) \leq r$.
Combining this with (\ref{GreavesCond5}), we obtain
\begin{equation}\label{GreavesR2}
\begin{aligned}
& \big|\big\{p\le  x : (|E(\F_p)|, M_E)=(|E(\F_p)|, P(D^V))=1, \;
\omega\big(|E(\F_p)|, D^U\big)\le r\big\}\big|
\\
& \geq B\frac{C_E^{\rm twin} x}{(\log x)^2}.
\end{aligned}
\end{equation}
When $(|E(\F_p)|, P(D^V))=1$,
we have $\omega\big(|E(\F_p)|, D^U\big) =
\Omega(|E(\F_p)|)$ unless $a$ is divisible by the square of a prime
$p$ such that $D^V \leq p < D^U$ and the required result
(\ref{GreavesR1}) now follows from (\ref{GreavesR2}) and hypothesis
(\ref{GreavesCond3}).
\end{proof}

\goodbreak
\smallskip

Now we are ready to prove Theorem \ref{CTEllipticCurve1}.

In Lemma \ref{Greaves1}, take

$$
r = r(\theta) := \bigg[\frac{18+2\theta}{5(1-\theta)}\bigg] + 1,
\quad
\xi=\frac{2(1-\theta)}{5} (1 - \varepsilon),
\quad
U=\frac{5}{8},
\quad
V=\frac{1}{4}
$$
where $\varepsilon$ is an arbitrary small positive number. It is
easy to see that condition (\ref{GreavesCond1}) is satisfied. In
order to verify (\ref{GreavesCond2}), it is sufficient to show that
$\xi (r U + V)>1$, since $a_p$ satisfies Hasse's bound
$|a_p|<2\sqrt{p}$. In view of the fact that $\varepsilon$ is
arbitrarily small, we have
$$
\xi (r U + V)
> \frac{2(1-\theta)}{5} (1 - \varepsilon)
\bigg(\bigg(\frac{18+2\theta}{5(1-\theta)} + \frac{8}{1-\theta}\varepsilon\bigg) \frac{5}{8} + \frac{1}{4}\bigg)
= 1 + \varepsilon.$$
It remains to verify (\ref{GreavesCond3}) and (\ref{GreavesCond4}).

First Theorem \ref{CDT-improved}(ii) allows us to deduce
\begin{align*}
\sum_{D^V\le p<D^U} |{\mathscr A}_{p^2}|
& \ll \sum_{D^V\le p<D^U} \bigg(\frac{x}{p^2\log x} + p^3 x^{\theta} \log x\bigg)
\\
& \ll D^{-V} x + D^{4U} x^{\theta}
\\\noalign{\vskip 2mm}
& \ll x^{1-\varepsilon (1-\theta)}
\end{align*}
for all $x\ge 3$.
This shows that (\ref{GreavesCond3}) is satisfied.

For $d$ square-free with $(d, M_E)= 1$, we can write
\begin{align*}
\left| \mathscr{A}_d \right|
& = \sum_{\substack{p \leq x\\ (|E(\F_p)|, M_E) = 1\\ |E(\F_p)|\equiv 0 \hskip -3mm \pmod d}} 1
\\
& = \sum_{m \mid M_E} \mu(m)
\sum_{\substack{p \leq x\\ |E(\F_p)|\equiv 0 \hskip -3mm \pmod d\\ |E(\F_p)|\equiv 0 \hskip -3mm \pmod m}} 1
\\
& = \sum_{m \mid M_E} \mu(m)
\sum_{\substack{p \leq x\\ |E(\F_p)|\equiv 0 \hskip -3mm \pmod {dm}}} 1
\\
& = \sum_{m \mid M_E} \mu(m) \pi_{C(dm)}(x, L_{dm}/\Q).
\end{align*}
Using Theorem \ref{CDT-improved}(i) and (\ref{ProbME}), we get that
\begin{align*}
\left| \mathscr{A}_d \right|
& = \Li(x) \frac{|C(d)|}{|G(d)|}
\sum_{m \mid M_E} \mu(m) \frac{|C(m)|}{|G(m)|}
+ O_E \left(|C(d)|^{1/2} x^\theta \log(dx)\right)
\\
& = \Li(x) \frac{|C(d)|}{|G(d)|} \left(1- \frac{|\Omega(M_E)|}{|G(M_E)|}\right)
+ O_E \left(d^{3/2} x^\theta \log(dx)\right).
\end{align*}
Thus we obtain
\begin{equation}
|{\mathscr A}_d|= \frac{w(d)}{d}X + r({\mathscr A}, d)
\end{equation}
for all $d\in {\mathcal B}({\mathscr P})$, with
$$w(d)
=\frac{d |C(d)|}{|G(d)|} = \prod_{\ell \mid d} \frac{\ell
(\ell^2-2)}{(\ell-1) (\ell^2-1)}
$$ and
$$ X:=\bigg(1-\frac{|\Omega(M_E)|}{|G(M_E)|}\bigg)\Li(x),
\qquad
|r({\mathscr A}, d)| \ll_E d^{3/2}x^{\theta}\log(dx).
$$
Since
$$
w(\ell) = 1 + \frac{\ell^2-\ell-1}{(\ell - 1)(\ell^2-1)},
$$
conditions $(A_0)$ and $(\Omega_1)$ are satisfied.
Thus Theorem A of \cite{HR85} is applicable. Denoting by $\gamma$
the Euler constant and defining
$$
V(D):= \prod_{\substack{p<D\\ p \in \mathscr{P}}}
\bigg(1-\frac{w(p)}{p}\bigg),
$$
Theorem A of \cite{HR85} allows us to write
\begin{equation}\label{GWeight}
\begin{aligned}
H\big({\mathscr A}, D^V, D^U\big)
& \ge \frac{2e^{\gamma}XV(D)}{U-V}
\bigg\{U\log\frac{1}{U}
+(1-U)\log\frac{1}{1-U}
-\log\frac{4}{3}
\\\noalign{\vskip 1mm}
& \quad
+ \alpha(V)-V\log 3-V\beta(V)
+ O\bigg(\frac{\log_3 D}{(\log_2D)^{1/5}}\bigg)\bigg\}
\\\noalign{\vskip 1mm}
& \quad
- (\log D)^{1/3}\Big|
\mathop{{\sum_{m<M}}\,{\sum_{n<N}}}_{mn\mid P(D^U)}
\alpha_m \beta_n r({\mathscr A},mn)\Big|
\end{aligned}
\end{equation}
where $M$ and $N$ are any two real numbers satisfying
$$M>D^U,
\qquad
N>1,
\qquad
MN=D$$
and $\alpha_m, \beta_n$ are certain real numbers satisfying $|\alpha_m|\le 1, |\beta_n|\le 1$.
The functions $\alpha(V)$ and $\beta(V)$ are given by
\begin{align*}
\alpha(V)
& := \log\frac{1-V}{(3/4)}
- \int_4^{1/V}\bigg(\frac{2}{u}\log(2-uV)+\log\frac{1-1/u}{1-V}\bigg) \frac{\log(u-3)}{u-2} \d u,
\\\noalign{\vskip 2mm}
\beta(V)
& := \log\frac{1-V}{3V}
- \int_4^{1/V}\bigg(\log(2-uV)+\log\frac{1-1/u}{1-V}\bigg)
\frac{\log(u-3)}{u-2} \d u,
\end{align*}
for $\frac{1}{6}\le V\le \frac{1}{4}$.

By using the prime number theorem, it follows that
\begin{equation}\label{V1}
\begin{aligned}
V(D) & = \prod_{\ell\le D, \, \ell\nmid M_E}
\bigg(1-\frac{|C(\ell)|}{|G(\ell)|}\bigg)
\\
& = \prod_{\ell\le D, \, \ell\nmid M_E} \bigg(1-\frac{1}{\ell}\bigg)
\prod_{\ell\le D, \, \ell \nmid M_E}
\bigg(1-\frac{|C(\ell)|}{|G(\ell)|}\bigg)\bigg(1-\frac{1}{\ell}\bigg)^{-1}
\\
& = \prod_{\ell \le D} \bigg(1-\frac{1}{\ell}\bigg) \prod_{\ell\mid
M_E} \bigg(1-\frac{1}{\ell}\bigg)^{-1} \prod_{\ell\le D, \,
\ell\nmid M_E}
\bigg(1-\frac{|C(\ell)|}{|G(\ell)|}\bigg)\bigg(1-\frac{1}{\ell}\bigg)^{-1}
\\
& \sim \frac{e^{-\gamma}}{\log D}
\prod_{\ell\mid M_E} \bigg(1-\frac{1}{\ell}\bigg)^{-1}
\prod_{\ell\nmid M_E} \bigg(1-\frac{\ell^2-\ell-1}{(\ell-1)^3(\ell+1)}\bigg)
\end{aligned}
\end{equation}
as $D\to\infty$.

On the other hand,
denoting by $\mu(d)$ the M\"obius function,
Theorem \ref{CDT-improved}(ii) implies that
\begin{equation}\label{Error}
\begin{aligned}
\Big|\mathop{{\sum_{m<M}}\,{\sum_{n<N}}}_{mn\mid P(D^U)}
\alpha_m \beta_n r({\mathscr A},mn)\Big|
& \le \sum_{d\le D}\mu(d)^23^{\omega(d)} d^{3/2} x^{\theta} \log(dx)
\\\noalign{\vskip -3mm}
& \ll x^{\theta} D^{5/2+\varepsilon (1-\theta)/2}
\\\noalign{\vskip 3mm}
& \ll x^{1-\varepsilon (1-\theta)/2}
\end{aligned}
\end{equation}
since $D=x^{\xi}$ with $\xi=\frac{2}{5}(1-\theta)(1-\varepsilon)$.

Combining (\ref{GWeight}), (\ref{V1}), (\ref{Error}) and (\ref{defCEtrwin}), we can find that
$$
H\big({\mathscr A}, D^V, D^U\big)
\ge \big\{2J(\xi, U, V)+o(1)\big\}\frac{C_E^{\rm twin}x}{(\log x)^2}
$$
with
\begin{align*}
J(\xi, U, V)
& := \frac{\alpha(V)-V\beta(V)-V\log 3-U\log U
-(1-U)\log(1-U)
-\log(4/3)}{\xi (U-V)}.
\end{align*}
Since $J(\xi, U, V)$ is continuous in $(\xi, U)$ and $\alpha(\frac{1}{4})=\beta(\frac{1}{4})=0$,
a simple numerical computation shows that
\begin{align*}
2J(\xi, U, V)
& = 2J\bigg(\frac{2(1-\theta)}{5}, \, \frac{5}{8}, \, \frac{1}{4}\bigg)+O(\varepsilon)
\\
& = \frac{1.32304\cdots}{1-\theta} + O(\varepsilon).
\end{align*}
This implies
$$
H\big({\mathscr A}, D^V, D^U\big)
\ge \frac{1.32303}{1-\theta}\frac{C_E^{\rm twin}x}{(\log x)^2}
$$
for $x\ge x_0(E)$. This completes the proof of Theorem
\ref{CTEllipticCurve1}. \qed

\section{Selberg's linear sieve and proof of Theorem \ref{Upperbound}}

We use the notation of Section \ref{GreavesLB}. As usual, we write
the sieve function
$$
S({\mathscr A}, {\mathscr P}, z)
:= |\{a\in {\mathscr A} : p\mid a \;\, {\rm and}\;\,  p\in {\mathscr P}\,\Rightarrow\,p>z\}|.
$$
Then in view of (\ref{Compare}), we can write the following trivial inequality
\begin{equation}\label{UB1}
\pi_E^{\rm twin}(x) \le S({\mathscr A}, {\mathscr P}, D^{1/2}) +
O(D^{1/2})
\end{equation}
for all $x\ge 1$, where $D=x^{\xi}$ with $\xi=2(1-\theta)(1 -
\varepsilon)/5$ as before, and the $O$-constant depends on the curve $E$.

Using Selberg's linear sieve \cite[Theorem 8.3]{HR74} with $q=1$, it
follows that
$$
S({\mathscr A}, {\mathscr P}, D^{1/2})
\le XV(D^{1/2}) \{F(2)+o(1)\} + {\mathscr R},
$$
where
\begin{align*}
{\mathscr R} & := \sum_{\substack{d<D\\ d\mid P(D^{1/2})}}
3^{\omega(d)} |r({\mathscr A}, d)|
\\
& \ll  \sum_{d<D} \mu(d)^2 3^{\omega(d)} d^{3/2} x^\theta\log(dx)
\\\noalign{\vskip 2mm}
& \ll x^{1-\varepsilon(1-\theta)/2},
\end{align*}
using Theorem \ref{CDT-improved}(i).

Since $F(2)=e^\gamma$,
replacing  $D$ by $D^{1/2}$ in the asymptotic formula (\ref{V1}), we
get that
\begin{align*}
XV(D^{1/2}) F(2)
& = C_E^{\rm twin} \frac{2 \Li(x)}{\log{x^\xi}} \left\{1 + o(1) \right\}
\\
& \le \bigg(\frac{5}{1-\theta} +
\varepsilon\bigg) \frac{C_E^{\rm twin}x}{(\log x)^2}
\end{align*}
for $x\ge x_0(E, \varepsilon)$.

Combining these estimates, we find that
\begin{equation}\label{UB2}
S({\mathscr A}, {\mathscr P}, D^{1/2})
\le \bigg(\frac{5}{1-\theta} + \varepsilon\bigg) \frac{C_E^{\rm twin}x}{(\log x)^2}
\end{equation}
for $x\ge x_0(E, \varepsilon)$. Theorem \ref{Upperbound} now follows
from (\ref{UB1}) and (\ref{UB2}). \qed

\end{document}